\documentclass[a4paper,12pt,final]{amsart}
\usepackage{times,a4wide,mathrsfs}

\newcommand{\C}{\mathbb{C}}

\newcommand{\QQ}{\mathbb{Q}}
\newcommand{\NN}{\mathbb{N}}

\newtheorem{theorem}{Theorem}

\newtheorem{lemma}{Lemma}
\newtheorem{corollary}{Corollary}
\newtheorem{proposition}{Proposition}
\newtheorem{conjecture}{Conjecture}
\newtheorem{remark}{Remark}
\newtheorem{definition}{Definition}
\newtheorem{convention}{Convention}

\begin{document}
\title[Various remarks about Casas-Alvero conjecture]
{Various remarks about Casas-Alvero conjecture}

\author[R. Laterveer, M. Ouna\"{\i}es]
{R. Laterveer \& M. Ounaies}

\address{Institut de Recherche Math\'ematique Avanc\'ee, Universit\'e 
Louis Pasteur 7 Rue Ren\'e Des\-car\-tes, 67084 Strasbourg CEDEX, France.}
\email{robert.laterveer@math.unistra.fr}

\address{Institut de Recherche Math\'ematique Avanc\'ee, Universit\'e 
de Strasbourg 7 Rue Ren\'e Des\-car\-tes, 67084 Strasbourg CEDEX, France.}
\email{myriam.ounaies@math.unistra.fr}

\date{\today}



\title{Constraints on hypothetical counterexamples to the Casas-Alvero conjecture}
\maketitle

In \cite{CA}, Casas-Alvero formulated the following :
\begin{conjecture}
Let $f\in \C \lbrack X \rbrack$ be a polynomial of degree $N\ge 1$. Assume that $f$ has a common root with each of its derivatives  $f^{(1)}, \cdots, f^{(N-1)}$ (we will say that $f$ is CA). Then $f$ is of the form $f(z)=a(z-b)^N$, $a,b\in \C$.
\end{conjecture}

The conjecture is known to be true in low degrees \cite{Diaz}. The first major break-through in the subject was \cite{Graf}, where the conjecture is proven for all $N$ that are powers of a prime number. A very nice overview of all that is currently known concerning this conjecture, as well as an alternative proof of the result of \cite{Graf}, can be found in \cite{Dr-Jo}. 

This note is the fruit of a failed attempt to prove the conjecture for $N=12$ (this is the lowest degree for which the conjecture has not been settled as yet). 

What follows is a short description of the contents of this note.

In the first section, which is largely expository, techniques are elementary. The degree $N$ is arbitrary, 
and we ask ourselves under what extra conditions on the number of roots the conjecture can be proven: 
The main new result in this first section is proposition \ref{4roots}: the conjecture is true if $f$ has at most 4 distinct roots.

In the second section we restrict attention to the case $N=p^r+1$ (where $p$ is a prime), and then further specialize to the case $N=p+1$. We use the $p$-adic valuation, inspired by Draisma and de Jong's take (\cite{Dr-Jo}, \cite{Jo}) on the result of Graf von Bothmer et alii \cite{Graf}. The main results of the second section are: If $f$ is a CA-polynomial of degree $N=p+1$, and $c$ is the root of $f^{(N-1)}$, then there are at least 2 other indices $2\le l_1<l_2<N-1$ such that $f^{(l_1)}(c)=f^{(l_2)}(c)=0$ (this is Proposition \ref{2fois}). Also, we are able to prove the conjecture for degree $p+1$ polynomials all of whose roots are rational (Proposition \ref{rational}).

In the third and last section, we present a short remark that gives some additional constraints on hypothetical counterexamples to the conjecture in degree 12.

We hope this note may raise interest in the subject, and perhaps spark further progress towards proving the conjecture...

\begin{convention}
In what follows, we will call a polynomial non-trivial if it has at least two distinct roots. The conjecture says that there exists no non-trivial CA-polynomial.
\end{convention}

From the definition of a CA-polynomial, it is easily seen that:
\begin{lemma}\label{change}
Let $f$ be a monic CA-polynomial of degree $N\ge 1$ and  $\alpha\in \C^*$ and $\beta\in \C$. The polynomial $g(z)=\alpha^{-N}f(\alpha z+\beta)$ is also a monic CA-polynomial of degree $N$.
\end{lemma}
In consequence, whenever convenient, we may choose any root of $f$ and assume it to be equal to $0$. When we have two distinct roots of $f$, we may assume the first to be equal to $0$ and the second to be equal to $1$. 

\section{Constraints on the number of distinct roots}

The following result, proven by Draisma and de Jong, conveniently subsumes many ``easy'' cases where the conjecture can be proven:

\begin{proposition}\cite[Proposition 6]{Dr-Jo}\label{Compl2}
Let $f$ be a CA-polynomial of degree $N\ge 1$. Let $\alpha_i$, $i=1,\cdots, N-1$ be a common root of $f$ and $f^{(i)}$. Then the cardinality of $\{\alpha_1,\cdots,\alpha_{N-1}\}$ cannot be two.
\end{proposition}

\begin{remark}
As immediate corollaries of Proposition  \ref{Compl2}, we get the following three special cases:
\item[1.] If $f$ is CA and $f''(z)=N(N-1)(z-a)^{N-2}$, then $f(z)=(z-a)^N$.
\item[2.] Let $f$ be a monic CA-polynomial of degree $N$, and suppose that $f$ has a root $a$ of multiplicity at least $N-1$. Then $f(z)=(z-a)^N$. 
\item[3.] Let $f$ be a non-trivial CA-polynomial, then $f$ has at least three distinct roots.

We note that case 1 was also dealt with in \cite[Proposition 2.2]{Verh}.  
\end{remark}

In the next two propositions, we will go a step further in directions 1 and  2. Later on (Propositions \ref{3roots} and \ref{4roots}), we will go two steps further in direction 3.

\begin{proposition}
Let $f$ be a CA-polynomial of degree $N\ge 4$. Assume that for some $a\in \C$, $f^{(3)}(z)=\frac{N!}{(N-3)!}(z-a)^{N-3}$. Then $f(z)=(z-a)^N$.
\end{proposition}

\begin{proof}
Thanks to Lemma \ref{change}, we may choose $a=0$. 

Assume that $f'(0)\not=0$, then $f$ has a root of multiplicity at least $2$ which is different from $0$ and again by Lemma\ref{change}, we may assume $f(1)=f'(1)=0$.
Thus
\[f(z)=z^N-(N-1)z^2+(N-2)z;\ \ f''(z)=(N-1)\left(Nz^{N-2}-2\right).\]
Solving $f(z)=f''(z)=0$, we get $z=\frac{N}{N+1}$ and  $(\frac{N+1}{N})^{N-2}=\frac{N}{2}$. We easily see that the function $\phi(t)=(t-2)\ln\frac{t+1}{t}-\ln\frac{t}{2}$ is strictly decreasing for $t\ge 4$ and that $\phi(4)<0$. Thus the equality $\phi(N)=0$ is never reached for $N\ge 4$. 
 
 We conclude that we necessarily have  $f'(0)=0$. Then, for some constant $c$, $f''(z)=N(N-1)z^{N-2}+2c$ and $f(z)=z^N+cz^2$. Solving $f(z)=f''(z)=0$, we get that $c=0$.
\end{proof}

\begin{proposition}\label{N-2}
Let $f$ be a monic CA-polynomial of degree $N\ge 3$, and suppose $f$ has a root $a$ of multiplicity at least $N-2$. Then $f(z)=(z-a)^N$.
\end{proposition}

\begin{proof}
Let $a=0$. If $f^{(N-1)}(0)\not=0$, then we may assume that  $f(1)=f^{(N-1)}(1)=0$. Hence 
\[f(z)=z^{N-2}(z^2-Nz+N-1),\ \ f^{(N-2)}(z)=\frac{(N-1)!}{2}(Nz^2-2Nz+2).\]
Solving $f(z)=f^{(N-2)}(z)=0$, we get $z^2=2$ and $z=\frac{N+1}{N}$. Thus $(N+1)^2=2N^2$ which is impossible.

 We conclude that we necessarily have  $f^{(N-1)}(0)=0$. Then, for some constant $c$, $f(z)=z^N+cz^{N-2}$ and $f^{(N-2)}(z)=\frac{N!}{2}z^2+c$. Solving $f(z)=f^{(N-2)}(z)=0$, we get $c=0$.
\end{proof}

\begin{remark}
We have chosen to present an elementary proof of Proposition \ref{N-2}, though we also can see it as direct consequence of the forthcoming Proposition \ref{3roots}.
\end{remark} 

Now, before going further, let us recall some basic properties of the elementary symmetric polynomials.
  Let a polynomial $f$ and its derivatives be of the form
\[f^{(l)}(z)=\frac{N!}{N-l!}  (a_0 z^{N-l}+{{N-l} \choose 1} a_1z^{N-l-1}+{{N-l} \choose 2} a_2z^{N-l-2}+\cdots+
a_{N-l}).\]
(Here by convention $f=f^{(0)}$).
Let $\sigma_m(l)$ be the sum of the $m$th powers of the roots of $f^{(l)}$, for $l=0,\cdots,N-1$. Then Newton formulas applied to each $f^{(l)}$ give the following relations (See for example \cite{Pr} for more details on Newton formulas):

\begin{lemma}\label{Relations}
\[\sum_{k=1}^j\sigma_k(l) {{N-l} \choose {j-k}}a_{j-k}=-j{{N-l} \choose j}a_j\]
for $0\le l\le N-1$, $1\le j\le  N-l$.

\end{lemma}

\begin{remark}  In particular, for $j=1$, we have that
\[\displaystyle \frac{\sigma_1(l)}{N-l}=\frac{\sigma_1(0)}{N}\]
 for  $l=0,\cdots,N-1$, which means that the center of mass of the roots of the derivatives is fixed. 
As obviously 
\[\displaystyle \sigma_1(N-1)=\frac{\sigma_1(0)}{N}=a_1\] is the only root of $f^{(N-1)}$, we see that whenever $f$ is a Casas-Alvero polynomial, the center of mass of its roots $\displaystyle \frac{\sigma_1(0)}{N}$ is itself a root of $f$. As a direct consequence,  the number of distinct roots of a Casas-Alvero polynomial cannot be two. Actually, we can say more: if $f$ has more than two distinct roots, then at least one of them (the center of mass) has to be in the interior of the convex hull of the roots. 
\end{remark}

In order to go further in the study of this convex hull, let us first recall the well-known Gauss-Lucas theorem:

\begin{theorem}[Gauss-Lucas]\label{Gauss-Luca} Let $f\in \C \lbrack X \rbrack $.
Let $f$ be a complex polynomial, $f'$ be its derivative and $\Gamma$ be the convex hull of the roots of $f$. Then each root of $f'$ is either a multiple root of $f$, or belongs to the interior of $\Gamma$. 
\end{theorem}

We will call the convex hull $\Gamma$ of the roots of $f$ the {\it Gaus-Lucas hull} of $f$, the interior of $\Gamma$  the {\it open Gauss-Lucas hull} and the boundary of $\Gamma$ the {\it Gauss-Lucas polygon}.

\begin{corollary}\label{sudbury}\cite{Su}
Let $f$ be a polynomial of degree $N$.
Let $\zeta$ be a root of $f$ located on its Gauss-Lucas polygon, with multiplicity $m\le N-1$.
Then, for all $m\le k\le N-1$, $f^{(k)}(\zeta)\not=0$.
\end{corollary}

This corollary yields another way of seeing that a non-trivial CA-polynomial $f$  must have at least one of its roots in its open Gauss-Lucas hull, and in particular that the number of distinct roots cannot be two.
Actually, we can also use Gauss-Lucas theorem to show that this number cannot be three either and more generally that a non-trivial CA-polynomial must have at least two distinct roots in its open Gauss-Lucas hull.

\begin{proposition}\label{3roots}
Let $f$ be a non-trivial Casas-Alvero polynomial of degree $N\ge 3$. Then $f$ has at least two distinct roots in its open Gauss-Lucas hull. In particular, $N\ge 5$ and $f$  has at least $4$ distinct roots.
\end{proposition}
\begin{proof}
Assume that $f$ has exactly one root, say $0$, in its open Gauss-Lucas hull and let $\zeta$ be the one among the roots of $f$ located on the Gauss-Lucas polygon with maximal multiplicity $m$. Then by Corollary\ref{sudbury}, $f^{(m)}(0)=f^{(m+1)}(0)=\cdots=f^{(N-1)}(0)=0$ which means that for $j=m,\ldots, N-1$:
\[\displaystyle f^{(j)}(z)=\frac{N!}{(N-j)!}z^{N-j}.\]

Taylor expansion gives 
\[f(0)=\sum_{j=m}^N\frac{f^{(j)}(\zeta)}{j!}(-\zeta)^{j}=\zeta^N\sum_{j=m}^N (-1)^j{N \choose j}=\zeta^N (-1)^m{{N-1} \choose m}.\]
As $f(0)=0$, we get $\zeta=0$, which is a contradiction.
\end{proof}

\begin{remark} Proposition \ref{3roots} can also be deduced from Draisma and de Jong's result Proposition \ref{Compl2}.
\end{remark}

As pointed out in \cite{Dr-Jo}, arguments based only on Gauss-Lucas theorem are not sufficient to deal with polynomials of degree greater than 5.  In the next proposition, we will use Rolle's theorem and relations \ref{Relations} to go further :  
\begin{proposition}\label{4roots}
Let $f$ be a non-trivial Casas-Alvero polynomial of degree $N\ge 5$. Then $f$ has at least five distinct roots. In 
particular, $N\ge 6$.
\end{proposition}

 \begin{proof}
 Assume that $f$ has four distinct roots. Then by the previous proposition, it  has at least two distinct roots in its open Gauss-Lucas hull. This implies that the four roots are on a line. By Lemma \ref{change}, we may assume that it is the real line. Then by Gauss-Lucas theorem, all its derivatives also have only real roots. 
  We will use the following lemma based on Rolle's theorem :
 \begin{lemma}\label{simple}\cite[Lemma 4.2]{Ko}
 A root of multiplicity $m$ of $f$, $0\le m\le i$ can be at most a simple root of $f^{(i)}$.
 \end{lemma}
 We denote by $M$ the maximal multiplicity of the roots of $f$. By Proposition \ref{N-2}, we know that  $2\le M\le N-3$.
\begin{itemize}
\item First case : $M\le N-5$. Again using Lemma \ref{change}, we may assume without loss of generality that the roots of $f$ are as follows : $a<0<1<b$ and $f^{(N-1)}(0)=0$.
By Corollary \ref{sudbury}, $a$ and $b$ cannot be zeros of $f^{(k)}$ for $N-5\le k\le N-1$. Besides, 
by Lemma \ref{simple}, each zero of $f^{(k)}$ is simple. Then we necessarily have $f^{(N-2)}(1)=0,\ f^{(N-3)}(0)=0,\ f^{(N-4)}(1)=0, \ f^{(N-5)}(0)=0$. Integrating five times the expression $f^{(N-1)}(z)=N!z$ and taking into account these constraints, we get $\displaystyle f^{(N-5)}(z)=\frac{N!}{5 !}z(z^2-5)^2$. But this contradicts Lemma \ref{simple}.

\item Second case : $M=N-4$ (which implies that $N\ge 6$). By Lemma \ref{change}, we arrange the roots as follows : $a<0<b<1$ and we assume that $f^{(N-1)}(0)=0$. Denote by $m_a,\ m_0,\ m_b,\ m_1$ their respective multiplicities. Then again by Lemma \ref{simple}, we must have $f^{(N-2)}(b)=0$, $f^{(N-3)}(0)=0$, $f^{(N-4)}(b)=0$. Like in the first case, computing the last derivatives, we get 
\[\begin{split}
&f^{(N-1)}(z)=N!z, \ \ \ 2! f^{(N-2)}(z)=N!(z^2-b^2),\\
&3! f^{(N-3)}(z)=N!z(z^2-3b^2),
 \ \ 4! f^{(N-4)}(z)=N!z(z^2-5b^2)(z^2-b^2).
\end{split}
 \]
Obviously, as $f^{(N-4)}(b)=0$, we have $m_b\le N-5$.
From Gauss-Lucas theorem, we deduce that $a<-\sqrt 5 b$. Now we apply Lemma \ref{Relations}  with $l=0$, $j=1$ and with $l=0$, $j=3$ to obtain
\begin{equation}\label{equ}
m_a a+m_b b+m_1=m_a a^3+m_b b^3+m_1=0.
\end{equation}
We deduce that $m_a a(a^2-1)=-m_bb(b^2-1)$ and looking at the sign, we see that $-a<1$. Then $m_a>-am_a=m_b b+m_1>m_1$ which implies that $m_a\ge 2$ and $m_1\le N-5$.

Let us note that in the case where  $m_a=2, m_1=m_b=1$, equations \ref{equ} give : $a(a+1)^2=0$. 

Thus, this case cannot occur. We can readily deduce that $m_0\le N-5$.
The only possibility left is $m_a=M=N-4$.

 From the relation
$-(N-4)a(1-a^2)=m_bb(1-b^2)$, we deduce that $\phi(-a)\le \phi(b)$ where we put $\phi(t)=t(1-t^2)$. But $\phi$ is increasing on $[0,1/\sqrt 3]$ and we know that $-a>b>0$. Thus we have $-a>1/\sqrt 3$. Now we get back to the linear equation in \ref{equ} : 
\[N-4=m_b\frac{b}{-a}+m_1\frac{1}{-a}<\frac{m_b}{\sqrt 5}+m_1\sqrt 3.\]
If $m_1=m_b=1$, this leads to $N=6$ and $m_a=2$. But we have already shown above that this case can't occur.

If $m_1=1, \ m_b=2$, we also obtain $N=6$ and $m_a=2$.

 Equations \ref{equ} now imply that $4a^2+2a-1=0$. But this equation has no real solution.

If $m_1=2, \ m_b=1$, we obtain $N\le 7$. On the other hand, we know that $m_1\le N-5$. Thus $N=7$ and $m_a=3$. 

Equations \ref{equ} imply $(a+1)^2(4a+1)=0$. Thus we must have $a=-1/4$ and $b=-5/4$. But this contradicts the fact that $b>0$.

Finally, the second case leads to a contradiction.

\item Third case : $M=N-3$. We proceed similarly to the previous case.  Again, we obtain  that $m_a\ge 2$. Thus we necessarily have : $m_a=M$, $m_0=m_1=m_b=1$. The linear equation in \ref{equ} gives : 
\[N-3=\frac{b}{-a}+\frac{1}{-a}<\frac{1}{\sqrt 5}+\sqrt 3<3.\]
It follows that $N=5$ and $m_a=2$. But we have already shown that this is impossible.
\end{itemize}

\end{proof}

\section{Polynomials of degree $p^r+1$}

We quickly recall the definition of the $p$-adic valuation (for more details, the reader is referred to \cite{Rib})

\begin{definition} Let $p$ be a prime number. For a positive integer $n\in\NN^*$, the $p$-adic valuation $v_p(n)$ is defined
to be the exponent of $p$ in the prime decomposition of $n$: If $n=p^r\cdot n^\prime$ with $n^\prime$ prime to $p$, then $v_p(n)=r$.
\end{definition}

The map $v_p\colon \NN^*\to \NN$ can be extended to a map $v_p\colon \C\to\QQ\cup\{+\infty\}$ (\cite[Chapter 4, Theorem 1]{Rib}). This map satisfies the following properties:

\begin{itemize}
\item $v_p(c)=+\infty$ if and only if $c=0$;
\item $v_p(ab)=v_p(a)+v_p(b)$ for all $a$, $b$ in $\C$;
\item $v_p(a+b)\ge \min\{ v_p(a), v_p(b) \}$ for all $a$, $b$ in $\C$.
\end{itemize}
 \begin{remark}
 It is important to note that the last property implies :
 $v_p(a+b)= \min\{ v_p(a), v_p(b) \}$ if $v_p(a)\not=v_p(b)$.
 
 We will make a frequent use of this fact.
 \end{remark}
 
\begin{lemma}\label{p^r}\cite[Lemma 2.4]{Graf}
Let $r\ge 1$ be an integer and $p$ be a prime number. Then $v_p({{p^r}\choose j})\ge 1$ for $1\le j\le p^r-1$.
\end{lemma}

From the relation $v_p({p^r+1\choose k})=v_p({{p^r}\choose {k-1}} +{{p^r}\choose k})$ we immediately deduce : 

\begin{lemma}\label{binom} Let $p$ be a prime, and suppose $N=p^r+1$. Then
\[ v_p({N\choose k})\ge 1\]
for $2\le k\le N-2$.
\end{lemma}


\begin{proposition}
Let $p$ be a prime number and let $f$ be a non-trivial  polynomial of degree $N=p^r+1$. Assume that $f$ has a common root with $f^{(2)},\cdots, f^{(N-1)}$. Let $c$ be the center of mass of $f$. Then the following conditions are satisfied :
\begin{itemize}
\item $f'(c)\not=0$, 
\item If $p\ge 3$, there is no  $w\in \C^*$ such that $f(c-w)=f(c+w)=0$,
\item If $p\ge 3$, there is no $w\in \C^*$ such that $f'(c-w)=f'(c+w)=0$.
\end{itemize}
\end{proposition}


\begin{proof}
We may assume without loss of generality, using Lemma \ref{change}, that $f$ is of the form

\[ f(z)=z^{N}+{N \choose 2} a_2z^{N-2}+\cdots+ {N \choose k} a_k z^{N-k}+\cdots+{N \choose {2}} a_{N-2}z^2+a_{N-1}z\] and that $\min\{v_p(z_j),\ j=1,\cdots,N\}=0$,
where we have denoted by $z_1,z_2,\cdots,z_{N-1}, z_N=0$ the zeros of $f$.

For $l=1,\cdots,N-1$, we have:

\begin{equation}\label{derivee1}
\frac{l!}{N!}f^{(N-l)}(z)=z^{l}+{l\choose 2} a_2z^{l-2}+\cdots+ {l\choose k} a_k z^{l-k}+\cdots+{l\choose {1}}a_{l-1}z+a_l.
\end{equation}
Following the proof of  \cite[Proposition 9]{Dr-Jo}, using equality (\ref{derivee1}) with $l=2,\cdots, N-2$ and $z$ the common root of $f^{(N-l)}$ and $f$, we prove by induction on $l$  that 
\begin{equation}\label{valuations1}
v_p(a_l)\ge 0 \ \ \ \hbox{ for all\ }  \ l=2,\cdots, N-2.
\end{equation}

Let $z_j$ be such that $v_p(z_j)=0$.  The equality
\[-a_{N-1}z_j=z_j^N+{N\choose 2} a_2z_j^{N-2}+\cdots+ {N\choose k} a_k z_j^{N-k}+\cdots+{N\choose {2}}a_{N-2} z_j^2\]
shows that $v_p(a_{N-1})=0$. In particular, $f'(0)\not=0$. Besides, as  $a_{N-1}=\prod_{j=0}^{N-1}z_j$,  we have $v_p(z_j)=0$ for $j\in\{1,\cdots,N-1\}$. 

If $p\ge 3$, we can also see that it is not possible to have a non-zero complex number $w$ such that $f(w)=f(-w)=0$. Indeed, otherwise, $0=f(w)-f(-w)$ would give the identity :
\[-a_{N-1}w={N\choose 3} a_3w^{N-3}+{N\choose 5} a_5w^{N-5}+\cdots+ {N\choose {3}} a_{N-3}w^3\]
which contradicts (\ref{valuations1}). 

We can use the same argument to show that there is no complex number $w$ such that $f'(w)=f'(-w)=0$. 
\end{proof}

\begin{remark} For $N=p+1$, it has been noted independently by Castryck that there is no non-trivial CA--polynomial of degree $N$ whose center of mass coincides with its double root \cite[Lemma 9]{Cast}.
\end{remark}

From now on, we will assume that $f$ is a non-trivial CA-polynomial of degree $N=p+1$ where $p\ge 3$ is a prime number. 

Using once again Lemma \ref{change}, we may assume that 

\begin{equation}\label{p+1}
\left\{
\begin{array}{ll}
&f(z)=z^{N}+N a_1z^{N-1}+{N\choose 2} a_2z^{N-2}+\cdots+ {N\choose k} a_k z^{N-k}+\cdots+{N\choose {2}} a_{N-2}z^2,\\
& \\
&N=p+1,\ \min\{v_p(z_j),\ j=1,\cdots,N\}=0,\ \ ,\end{array}
\right.
\end{equation}
where we have denoted by $z_1,\cdots, z_{N-3}, z_{N-2}=z_{N-1}=0, z_N=-a_1$ the roots of $f$.

For $l=1,\cdots,N-1$, we have:

\begin{equation}\label{derivee2}
\frac{l!}{N!}f^{(N-l)}(z)=z^{l}+{l\choose 1} a_1z^{l-1}+{l\choose 2} a_2z^{l-2}+\cdots+ {l\choose k} a_k z^{l-k}+\cdots+{l\choose {1}}a_{l-1}z+a_l.
\end{equation}
Observe that $v_p(a_1)\ge 0$ because $-a_1$ (the center of mass of $f$) is one of the roots of $f$.
Once again following the proof of \cite[Proposition 9]{Dr-Jo} and using equality (\ref{derivee2}) with $l=2,\cdots, N-2$ and $z$ the common root of $f^{(N-l)}$ and $f$, we prove by induction on $l$  that 
\begin{equation}\label{valuations2}
v_p(a_l)\ge 0 \ \ \ \hbox{ for all\ }  \ l=2,\cdots, N-2.
\end{equation}

Let $z_j$ be such that $v_p(z_j)=0$.  The equality
\[-N a_1 z_j^{N-1}= z_j^N+{N\choose 2} a_2z_j^{N-2}+\cdots+ {N\choose k} a_k z_j^{N-k}+\cdots+{N\choose {2}}a_{N-2} z_j^2\]
shows that $v_p(a_1)=0$. Therefore, we may assume without loss of generality that $a_1=-1$. Then we can write
$f(z)=(z-1)g(z)$ where
\[
\begin{split}
g(z)&=\Bigl(z^{N-1}-(N-1)z^{N-2}+({N\choose 2} a_2-(N-1))z^{N-3}+\\
&({N\choose 3} a_3+{N\choose 2}a_2-(N-1))z^{N-4}+\cdots+({N\choose 3}a_{N-3}+\cdots+{N\choose {2}} a_2-(N-1))z^2\Bigr).
\end{split}\]
In view of (\ref{valuations2}) and Lemma \ref{binom}, all roots of $g$ have strictly positive valuations  (actually greater than $1/(N-3)$). As a consequence, we see that $1$ is a simple root of of $f$ and that $v_p(z_j)>0$ for $j=1,\cdots,N-3$.

Whenever $f^{(N-l)}(1)\not=0$, the Casas-Alvero property implies that $f^{(N-l)}(z_j)=0$ with $v_p(z_j)>0$ and from equality (\ref{derivee2}) we get  $v_p(a_l)>0$.

But as 
\[ f(1)=1-N+{N\choose 2}a_2+\cdots+{N\choose k} a_k+\cdots+{N\choose {2}} a_{N-2}=0,\] 
there is at least one index $2\le l\le N-2$ such that $v_p(a_l)=0$. In other words, at least one of the derivatives $f^{(N-l)}(1)=0$. If we put this result together with Proposition \ref{Compl2}, we get :

\begin{lemma}\label{1fois}
Let $f$ be a non-trivial CA-polynomial of degree $N=p+1$, $p$ prime. Let $c$ be the center of mass of $f$. Then the  following conditions are satisfied :
\begin{itemize}
\item $f^{(N-1)}(c)=0$,
\item $f^{(l)}(c)\not=0$ for at least one $l\in \{2,\cdots,N-2\}$,
\item $f^{(k)}(c)=0$ for at least  one $k\in \{2,\cdots,N-2\}$.
\end{itemize}
\end{lemma}

Let us now go further into the investigation of the orders of the derivatives having the center of mass as a root.

We may again assume that $f$ is of the form (\ref{p+1}) and that $a_1=-1$.

For the sake of clarity, we will use the notation $x\equiv y$ if $v_p(x-y)>0$.

In view of Lemma \ref{1fois}, let $\ l_1<l_2<\cdots<l_m\ $ be the indices between $2$ and $N-2$ such that 
$f^{(N-l_j)}(1)=0$, $j=1,\cdots,m$. 

As observed previously, for all $k\in\{2,\cdots,N-2\}$, we have $v_p(a_k)\ge0$. Moreover, if $k\notin \{l_1,\cdots, l_m\}$, $a_k\equiv 0$.

From equality (\ref{derivee2}) with $z=1$ and $l=l_1, l_2\cdots, l_m$, we get 
\begin{equation}\label{S}
\left\{
\begin{array}{ll}
1-l_1+a_{l_1} & \equiv 0\\
1-l_2+\binom{l_2}{l_1}a_{l_1}+a_{l_2} & \equiv 0\\
\ \ \ \ \  \vdots\\
1-l_m+\binom{l_m}{l_1} a_{l_1}+\binom{l_m}{l_{2}}a_{l_2} +\cdots+ a_{l_m} & \equiv 0\\
\end{array}
\right.
\end{equation}
Now, using that  $\frac{f(1)}{p}=0$ and that $v_p(\binom{k}{N})\ge 1$ for $k=2,\cdots, N-2$, we obtain
\begin{equation}\label{f(1)=0}
-1+\frac{\binom{N}{l_1}}{p} a_{l_1}+\cdots+\frac{\binom{N}{l_m}}{p}a_{l_m} \equiv 0.
\end{equation}
 Now observe that
 for all $2\le l\le N-2$, we have :
\[\begin{split}
\frac{{N\choose l}}{p}&=\frac{N(N-2)(N-3)\cdots(N-(l-1))}{l!}\\
&=\frac{(p+1)(p-1)(p-2)\cdots(p-(l-2))}{l!}\\
&=\frac{1}{l!}(p^{l-1}+\alpha_{l-2}p^{l-2}+\cdots+
\alpha_1p)+\frac{(-1)^{l-2}(l-2)!}{l!}\end{split}\]
	where $\alpha_1,\cdots, \alpha_{l-2}$ are integers. 

Therefore: \[ \frac{\binom{N}{l}}{p}\equiv \frac{(-1)^l}{l(l-1)}.\]

Putting equations (\ref{S}) and (\ref{f(1)=0}) together and putting $\tilde a_{l_j}=\frac{a_{l_j}}{l_j(l_j-1)}$, we obtain:

\begin{equation}\label{systeme}
\left\{
\begin{array}{ll}
-1+l_1\tilde a_{l_1} & \equiv 0\\
-1+\binom{l_2-2}{l_1-2}l_2 \tilde a_{l_1}+l_2\tilde a_{l_2} & \equiv 0\\
\ \ \ \ \  \vdots\\
-1+\binom{l_m-2}{l_1-2}l_m \tilde a_{l_1}+\binom{l_m-2}{l_{2}-2}l_m \tilde a_{l_2} +\cdots+ l_m \tilde a_{l_m} & \equiv 0\\
-1+(-1)^{l_1} a_{l_1}+(-1)^{l_2}\tilde a_{l_2}+\cdots+(-1)^{l_m}\tilde a_{l_m} & \equiv 0.
\end{array}
\right.
\end{equation}

Let us define the determinant  
\begin{equation}\label{Delta}
\Delta=\left[
\begin{array}{c c c  c  c c }
-1& l_1 & 0  & 0 &\cdots  & 0\\
-1 & \binom{l_2-2}{l_1-2}l_2 & l_2  & 0 &\cdots & 0 \\
\vdots & \vdots & \vdots  &\vdots & & \vdots \\
-1 & \binom{l_m-2}{l_1-2}l_m & \binom{l_m-2}{l_{2}-2}l_m & &\cdots & l_m \\
-1 & (-1)^{l_1} & (-1)^{l_2} & & \cdots  & (-1)^{l_m}
\end{array}
\right].
\end{equation}

We see that necessarily $\Delta=0$. Otherwise, inverting (\ref{systeme}), we would get in particular that $1\equiv 0$. Let us summarize this result by: 

\begin{lemma}\label{DeltaL}
Let $f$ be a Casas-Alvero polynomial of the form (\ref{p+1}) with $a_1=-1$. Let $\ l_1<l_2<\cdots<l_m\ $ be the indices between $2$ and $N-2$ such that 
$f^{(N-l_j)}(1)=0$, $j=1,\cdots,m$ and let $\Delta$ the determinant defined by (\ref{Delta}). Then
$p$ divides $\Delta$.
\end{lemma}

\
\begin{proposition}\label{2fois}
Let $f$ be a non-trivial  CA-polynomial of degree $N=p+1$, $p$ prime. Then there are at least two indices $2\le l_1<l_2\le N-2$ such that 
 $f^{(l_1)}(c)=f^{(l_2)}(c)=0$.
\end{proposition}
\begin{proof}
If not, in virtue of Lemma \ref{1fois}, there exists a unique index $2\le l\le N-2$ such that $f^{(N-l)}(c)=0$.

We can assume without loss of generality that $f$ is of the form (\ref{p+1}) with $a_1=-1$ and apply Lemma \ref{DeltaL}. Then $m=1$ and 
\begin{equation}
\Delta=\left[
\begin{array}{l l }
-1& l \\
-1 & (-1)^{l}
\end{array}
\right]=l-(-1)^l.
\end{equation}
Observe that $1\le l-(-1)^l\le l+1\le N-2$ for $l\in{2,\cdots, N-3}$. 

\noindent Besides, $N-2-(-1)^{N-2}=N-3$ because $N$ is even. 

\noindent Finally, there is no way for $p$ to divide $\Delta$ and this contradicts Lemma \ref{DeltaL}.
\end{proof}

\begin{remark} We can actually go a bit further, but the results are not as conclusive. For instance, let $f$ be a CA-polynomial of degree $N=12$, and suppose
there are exactly two indices $2\le l_1<l_2<N-1$ such that $f^{(l_1)}$, $f^{(l_2)}$ and $f^{(N-1)}$ have a common root. Then applying Lemma \ref{DeltaL}, we can check that there are only $4$ possibilities for $(l_1,l_2)$: $(l_1,l_2)$ must be $(3,8)$, $(5,6)$, $(6,8)$ or $(6,9)$.
\end{remark}

\begin{remark} In \cite{Cast}, Castryck defines the \emph{type} of a CA-polynomial $f$ of degree $N$ as follows: Let $\{ \alpha_1,\ldots,\alpha_k\}$ denote the roots of $f$. A subset $S\subset \{\alpha_1,\ldots,\alpha_k\}$ is called \emph{covering} if for each $i=1,\ldots,N-1$, there exists $\alpha\in S$ such that $f^{(i)}(\alpha)=0$.
The type of $f$ is defined as
  \[ type(f)= \min\{\hbox{Card } S-1,\ S \hbox{ covering}\}.\]
 Castryck establishes that 
 \[2\le type(f)\le N-3\] 
 for any non-trivial CA-polynomial $f$ (\cite{Cast}; the lower bound follows from Proposition \ref{Compl2}). Our result Proposition \ref{2fois} gives something more: any CA-polynomial $f$ of degree $N=p+1$ satisfies
 \[ type(f)\le N-4.\]
 In particular, for $N=12$, this bound combined with the results of Castryck (loc. cit.) implies that in order to prove the Casas-Alvero conjecture in degree $12$, the only cases remaining to be checked are polynomials of type $6$, $7$ and $8$.
 \end{remark}
Our final result proves the Casas-Alvero conjecture for degree $p+1$ polynomials whose roots are rational numbers:

\begin{proposition}\label{rational}
There is no non-trivial CA-polynomial of degree $N=p+1$, $p\ge 3$ prime, with rational roots. 
\end{proposition}

\begin{proof}
Using the same notations as in the  proof of Lemma \ref{p+1}, we may assume that $f$ is of the form
\[
\begin{split}
f(z)=&z^{N}-N z^{N-1}+{N\choose 2}z^{N-2}\\
    &+\cdots+(-1)^{k-1}{N\choose {k-1}}z^{N-k+1}+{N\choose k}a_kz^{N-k}+\cdots+{N\choose 2} a_{N-2}z^2,
    \end{split}\] 
    with
 $v_p(z_j)\ge 1$ for $j=1,\cdots, N-3$. Here, we have denoted by $k$ the smallest index between $2$ and $N-2$ such that $f^{(N-k)}(1)\not=0$ (we know from Lemma \ref{p+1} such a $k$ exists).

 We introduce the notation 
 \[S_m=\sum_{j=1}^{N-3} z_j^m.\] 
 Then we have: $v_p(S_1)=v_p(N-1)=1$, and $v_p(S_j)\ge 2$ for $j=2,\cdots,N-2$. Using Newton formulas (see Lemma \ref{Relations}, $l=0$), we obtain
\[
\begin{split}
 -k{N\choose k}a_k&=\sum_{j=0}^{k-1}(-1)^j(1+S_{k-j}){N\choose j}\\   &=\sum_{j=0}^{k-1}(-1)^j{N\choose j}+\sum_{j=0}^{k-1}(-1)^jS_{k-j}{N\choose j}\\
& =(-1)^{k-1}{{N-1}\choose {k-1}}+\sum_{j=0}^{k-1}(-1)^jS_{k-j}{N\choose j}.\\
 \end{split}\]
 Note that $v_p({N\choose k}a_k)>1$ which will lead to a contradiction:
 
 Case 1: $k=2$. The last equality becomes :
 
 \[-2{N\choose 2} a_2=-(N-1)+S_2-NS_1=-(N-1)+S_2-N(N-1)=-(N+1)(N-1)+S_2.\]
The valuation of the right-hand term is $1$. 
 
 Case 2: $3\le k\le N-2$ : The right-hand term is
 \[(-1)^{k-1}{{N-1}\choose {k-1}}+\sum_{j=0}^{k-2}(-1)^jS_{k-j}{N\choose j}+(-1)^{k-1}S_1{N\choose {k-1}}.\]
 But $v_p(S_{k-j})\ge 2$ for $j=0,\cdots,k-2$, and $v_p(S_1{N\choose {k-1}})=2$, so the valuation of the right-hand term
 is $v_p({{N-1}\choose {k-1}})=1$.
 
 \end{proof}
\section{A final remark about a possible counterexample in degree 12}
  
\begin{proposition}\label{12}
Let $f$ be a non-trivial CA-polynomial of degree $N\ge 1$.  
\begin{itemize}
\item [1)] 
Assume that some prime $q\ge 2$ divides all binomial coefficients $\binom{N}{k}$, $k=1,\cdots,N-1$ except for $k=q^{s}$ and $k=N-q^s=q^{t}$, $s,t\in \NN$. Then the derivatives $f^{(q^{s})}$ and $f^{(N-q^{s})}$ don't have any common root. 
\item [2)] Assume that some prime $q\ge 2$ divides all binomial coefficients $\binom{N}{k}$, $k=1,\cdots,N-1$ except for $k=l_1,\cdots,l_m$. Then there is no $a\in \C$ such that 
\[f(a)=f^{l_1}(a)=\cdots=f^{l_m}(0).\] 
\end{itemize}
\end{proposition}
\begin{remark}
1)
This implies in particular that $s\not=t$. This means that there doesn't exist a no-trivial CA-polynomial with degree $N=2q^r$, $q$ prime. But this result is already proved in \cite{Graf}. 
\end{remark}

\begin{proof}

\begin{itemize}

\item [1)] 
We may assume that  $s\le t$ and that  $f$and its derivatives are of the form : 
 
 \begin{equation}\label{derivee3}
\frac{l!}{N!}f^{(N-l)}(z)=z^{l}+{l\choose 1} a_1z^{l-1}+\cdots+{l\choose {1}}a_{l-1}z+a_l,
\end{equation}
with $a_N=a_{q^s}=0$ and  $\min\{v_q(z_j),\ j=1,\cdots,N\}=0$  where we have denoted by $z_1,\cdots, z_N$ the roots of $f$.

As in \cite[Proposition 9]{Dr-Jo}, using equality (\ref{derivee3}) with $l=1,\cdots, N-1$ and $z$ the common root of $f^{(N-l)}$ and $f$, we prove by induction on $l$  that 
\begin{equation}\label{valuations1}
v_q(a_l)\ge 0 \ \ \ \hbox{ for all\ }  \ l=1,\cdots, N.
\end{equation}

Using $f(z_j)=0$ with $v_q(z_j)=0$, we deduce that $v_q(a_{q^t})=0$. 

Now we apply formula (\ref{derivee3}) with $l=q^s$ then $l=N-q^s$. We find, with the help of  Lemma \ref{p^r}, that all roots of $f^{(N-q^s)}$ have q-valuation $>0$ while all roots of $f^{(q^s)}$ have q-valuation $=0$. 

\item [2)] We proceed by contradiction. Let us repeat the beginning of the proof in 1) but this time assuming that $a_N=a_{l_1}=\cdots=a_{l_m}=0$. We still get 
\begin{equation}\label{valuations1}
v_q(a_l)\ge 0 \ \ \ \hbox{ for all\ }  \ l=1,\cdots, N.
\end{equation}
Using $f(z_j)=0$ with $v_q(z_j)=0$ we have
\[-z_j^{N}={N\choose 1} a_1z_j^{N-1}+\cdots+{N\choose {N-1}}a_{N-1}z_j.\]
But by assumption, the q-valuation of the right-hand term is $\ge 1$  and we reach a contradiction.
\end{itemize}

\end{proof}

\begin{remark}[Case N=12]
Observe that $2$ (respectively $3$) divides $\binom{12}{k}$, $k=1,\cdots,11$, except for $k=4, 8$ (resp. k=3,9). Thus, a possible counterexample in degree $12$ should satisfy : $f^{(4)}$ and $f^{(8)}$ (respectively $f^{(3)}$ and $f^{(9)}$) don't share any root. 


\end{remark}

\end{document}